\documentclass[review]{elsarticle}

\usepackage{lineno,hyperref}
\usepackage{amsmath, amsthm, amssymb}

\newtheorem{problem}{Problem}
\usepackage{graphicx}
\usepackage{setspace}
\usepackage{enumitem,verbatim}
\usepackage{booktabs}
\newcommand{\R}{\mathbb{R}}           
\newcommand{\nth}{^{\text{th}}}       
\newtheorem{theorem}{Theorem}
\newtheorem{lemma}{Lemma}
\newtheorem{proposition}{Proposition}
\newtheorem{definition}{Definition}
\newtheorem*{definition*}{Definition}
\journal{Discrete Mathematics}

\begin{document}
	
	\begin{frontmatter}
		
		\title{Query Complexity of Mastermind Variants}

		\author[aaron]{Aaron Berger\corref{mycorrespondingauthor}}
		
		\author[chris]{Christopher Chute}\author[matt]{Matthew Stone}
		\address{Yale University, 10 Hillhouse Ave, New Haven, CT 06511}
	\fntext[aaron]{aaron.berger@yale.edu}\fntext[chris]{christopher.chute@yale.edu}\fntext[matt]{matthew.i.stone@yale.edu}
		

		\cortext[mycorrespondingauthor]{Corresponding author}

		\begin{abstract}
				We study variants of Mastermind, a popular board game in which the objective is sequence reconstruction. In this two-player game, the so-called \textit{codemaker} constructs a hidden sequence $H = (h_1, h_2, \ldots, h_n)$ of colors selected from an alphabet $\mathcal{A} = \{1,2,\ldots, k\}$ (\textit{i.e.,} $h_i\in\mathcal{A}$ for all $i\in\{1,2,\ldots, n\}$). The game then proceeds in turns, each of which consists of two parts: in turn $t$, the second player (the \textit{codebreaker}) first submits a query sequence $Q_t = (q_1, q_2, \ldots, q_n)$ with $q_i\in \mathcal{A}$ for all $i$, and second receives feedback $\Delta(Q_t, H)$, where $\Delta$ is some agreed-upon function of distance between two sequences with $n$ components. The game terminates when the codebreaker has determined the value of $H$, and the codebreaker seeks to end the game in as few turns as possible. Throughout we let $f(n,k)$ denote the smallest integer such that the codebreaker can determine any $H$ in $f(n,k)$ turns. We prove three main results: First, when $H$ is known to be a permutation of $\{1,2,\ldots, n\}$, we prove that $f(n, n)\ge n - \log\log n$ for all sufficiently large $n$. Second, we show that Knuth's Minimax algorithm identifies any $H$ in at most $nk$ queries. Third, when feedback is not received until all queries have been submitted, we show that $f(n,k)=\Omega(n\log k)$.
		\end{abstract}
		
		
	\end{frontmatter}
	
	\section{Introduction}\label{Intro}
	The original 1970 conception of Mastermind by Mordechai Meirowitz was a sequence reconstruction game.  One player (the \emph{codemaker}) would construct a hidden sequence of four pegs, each peg being one of six colors, and the other player (the \emph{codebreaker}) would make guesses of the same form, receiving feedback after each guess regarding how close they were to the hidden sequence. In 1963 Erd\H os and R\'enyi \cite{ER63} studied the two-color variant of this game, and after the release of Mastermind, Knuth showed that a minimax strategy guarantees guessing the hidden vector in no more than 5 turns \cite{DK76}. Many authors have since studied algorithms to minimize the number of guesses required in the worst case \cite{DS13,VC83,KT86,MG12,OS13,JP11,CC96,JP15, WG03, WG04}, and almost all of these results will be introduced and discussed at relevant points in this paper. We note here that the work of J{\"a}ger and Peczarski \cite{JP11,JP15} and Goddard \cite{WG03,WG04} deal with finding explicit optimal bounds for small numbers of colors and pegs, whereas we deal with asypmtotics when both of these quantities are large.
	
	The variants of Mastermind which we study are defined by the following parameters:
	\begin{enumerate}[label=(\roman*)]
		
		\item ($k$) \textit{Size of Alphabet} 
		
		\item ($n$) \textit{Length of Sequence.} The hidden vector and all guess vectors will be elements of $[k]^n$.
		
		\item ($\Delta$) \textit{Distance Function.} $\Delta$ takes as inputs two vectors in $[k]^n$. The output may, for example, be a single integer, but this will not always be the case. Most research studies the two following distance functions:
		\begin{enumerate}[label=\alph*.]
			\item\textit{``Black-peg and white-peg.''} Informally, a black peg denotes ``the correct color in the correct spot,'' and a white peg denotes ``the correct color in an incorrect spot.'' For two vectors $Q_t$ and $H$, the black-peg and white-peg distance function is the ordered pair $\Delta(Q_t, H) := (b(Q_t, H), w(Q_t, H))$ where
			\begin{equation*}\label{blackHitsDefinition}
				b(Q_t, H) = \left|\{i\in [1,n] \mid q_i = h_i\}\right|,
			\end{equation*}
			and
			\begin{equation*}\label{whiteHitsDefinition}
				w(Q_t, H) = \max_{\sigma}~b(\sigma(Q_t), H) - b(Q_t, H),
			\end{equation*}
			where $\sigma$ iterates over all permutations of $Q_t$. This variant is the distance function used in the original game of Mastermind.
			
			\item\textit{``Black-peg-only.''} This is simply $\Delta(Q_t, H) := b(Q_t, H)$, where $b$ is defined as above.
		\end{enumerate}
		
		\item($R$) \textit{Repetition.} A commonly-studied variant of the game introduces the restriction that the guesses and vectors cannot have repeated components, i.e. they are vectors of the form $v \mid i \neq j \Rightarrow v_i \neq v_j$.
		
		\item($A$) \textit{Adaptiveness.} In the \textit{adaptive} variant of Mastermind, the codebreaker receives $\Delta(Q_t, H)$ after each guess $Q_t$, and may use this information to inform the selection of $Q_{t+1}$.
		
		In the \textit{non-adaptive} variant, the codebreaker submits any number of queries $Q_1, Q_2, \ldots, Q_m$ all at once (so the codebreaker chooses $m$). The codemaker then reports the distances $(\Delta(Q_1, H), \Delta(Q_2, H), \ldots, \Delta(Q_m, H))$, after which the codebreaker must determine $H$ without submitting any additional guesses.
	\end{enumerate}
	
	The adaptive variant with repetitions allowed is the most extensively studied in the literature \cite{DK76,VC83,MG12,CC96}. Doerr, Sp\"{o}hel, Thomas, and Winzen obtain both strong asymptotic lower bounds and an asymptotic improvement in algorithm performance \cite{DS13}, in part applying techniques from \cite{NB09,GK00}. 
	
	We focus only on analyzing the worst-case performance of query strategies for these variants of Mastermind. That is, we always consider the number of queries necessary to guarantee identification of any hidden vector. As such, we will be only be focusing on deterministic strategies for the codebreaker.
	
	Our first main result concerns the \textit{Permutation Variant}, the black-peg, adaptive, no-repeats variant in which $n=k$. In this variant, the hidden sequence $H$ and all guesses $Q_t$ are permutations of $[n]$, hence the name.
	\begin{theorem}\label{permutationGameTheorem}
		For any strategy in the Permutation Variant with $n$ sufficiently large, the codebreaker must use at least $n - \log\log n$ guesses to determine $H$ in the worst case.

	\end{theorem}
	
	Explicit algorithms that take $O(n \log n)$ turns to solve this variant were developed by Ko and Teng \cite{KT86}, and El Ouali and Sauerland \cite{OS13}. Ko and Teng approach the problem with an algorithm akin to binary search. El Ouali and Sauerland improve this algorithm and extend it to handle variants with $k\ge n$, while also achieving an average factor of two reduction in the number of queries needed to identify $H$.
	
	Via a basic information-theoretic argument, one can show that the Permutation Game satisfies $f(n, n)\ge n - n/\log n + c$ for some constant $c>0$. We improve this lower bound to $f(n, n)\ge n- \log \log n$ for sufficiently large $n$. To our knowledge, this constitutes the first improvement over the trivial information-theoretic lower bound for the Permutation Game variant of Mastermind.
	
	Our second result concerns Knuth's Minimax algorithm for adaptive variants, which was first introduced in 1976:
	\begin{definition}[Knuth's Minimax Algorithm \cite{DK76}]\label{minimax def}
		At each turn, assign each query a score equal to the maximum across all responses of the number of possible values of $H$ that agree with that response. Guess the query with the minimum score.
		\end{definition}
	\begin{theorem}\label{minimaxTheorem}
		Knuth's Minimax algorithm identifies any hidden sequence $H$ in at most $nk$ queries.
	\end{theorem}
	Knuth's Minimax algorithm is empirically near-optimal for solving small games of Mastermind ($n$ and $k$ less than 10) in as few guesses as possible. However, it has proven difficult to analyze the asymptotic performance of the Minimax algorithm, primarily because its behavior is determined by the distribution remaining solutions after a series of guesses, which is difficult to analyze in general \cite{ KT86, OS13}. To our knowledge, this is the first upper bound on the worst-case performance of the minimax algorithm, but if it performs near-optimally for large $n$ and $k$, we would expect this bound to be much smaller. We know, for example, that algorithms exist that use only $O(n \log k)$ guesses when $k$ is not too large \cite{DK76, KT86}.
	
	Our third result concerns to non-adaptive variants of Mastermind. We extend the following theorem:
	\begin{theorem}[Doerr, Sp\"ohel, Thomas, and Winzen, \cite{DS13}]\label{DSTW nonadapt}
		In black-peg, non-adaptive Mastermind with repeats,
		$
		\Omega\left( n\log(k)\right)
		$
		guesses are required to identify $H$,
	\end{theorem}
	proving the result in the no-repeats case.
	\begin{theorem}\label{nonAdaptiveTheorem}
		In black-peg, non-adaptive Mastermind with no repeats,
		$
			\Omega\left( n\log(k)\right)
		$
		guesses are required to identify $H$.
	\end{theorem}
	For neither variant do the known upper bounds match these lower bounds when $n > k$; in the with-repeats case a corollary of a result in \cite{DS13} gives an upper bound of $O(k \log k)$ guesses, and for the no-repeats case no improvement over the $nk$ bound is known. On the other hand, the authors of \cite{DS13} are able to extend a result of Chv\'atal \cite{VC83} to provide tight bounds for $n \le k$ with in the with-repeats case.
	\subsection{Structure of the Paper}
	
	In Section 2 we prove Theorem 1, and the proof is found in 2.2. In Section 3 we prove Theorem 2, and in Section 4 we discuss Theorem 3, with the proof in 4.2.
	\section{Adaptive Variants of Mastermind}
	\subsection{The Permutation Game}
	We begin with notation necessary for the proof of Theorem 1. Recall that this is the adaptive, no-repeats variant with $n=k$. White pegs provide no information in this variant, so without loss of generality we assume black-peg responses. Throughout this section we will let $f(n)$ be the number of guesses required by an optimal strategy for the permutation game on $[n]^n$. We will bound $f(n)$ from below.
	
	We will use the derangement function $D(n)$, which counts the number of permutations in $S_n$ with no fixed points. It has the explicit form
	\begin{equation}\label{derangements}
		D(n) = n!\sum_{i=0}^n \frac{(-1)^i}{i!},
	\end{equation}
	and is the nearest integer to $n!/e$.  
	\subsubsection{Trivial Lower Bound}
	To motivate the proof of Theorem 1, we begin with the following simple result.
	\begin{proposition}[Trivial Lower Bound]
		\begin{equation*}
			f(n)\ge \log_{n}(n!) = n-\frac{n}{\ln(n)}+O(1).		
		\end{equation*}
	\end{proposition}
	\begin{proof}
		In a deterministic strategy with $t$ queries, there are $n^t$ possible responses (specifically, $\{0,1,\ldots,n-2,n\}^t$). When $t < \log_{n}(n!)$ there are fewer possible responses than possible values of $H$, so by the pigeonhole principle at least two distinct values of $H$ will produce the same set of responses, and the codebreaker will be unable to distinguish between the two.
	\end{proof}
	
	\subsubsection{Solution subsets}\label{Buckets Definition}
	We now improve this trivial bound. Assume the codebreaker has some arbitrary fixed, deterministic guessing strategy. In $t$ turns of the game, the codebreaker has submitted queries $Q_1, \dots, Q_t$ and received responses $r_1,\dots,r_t$. The set of solution vectors $h$ satisfying $\Delta(Q_i,h) = r_i$ for all $i$ will be called the \textit{remaining solution set} $S_t$, and the codebreaker wins exactly when $|S_t| = 1$.
	
	Now we analyze the performance of the codebreaker's strategy. On turn $t$, the codebreaker makes the guess $Q_t$ according to some deterministic procedure. Each of the $n$ possible responses $\Delta(Q_t,H)$ produces a different remaining solution set $S_t$. Moreover, these different choices of $S_t$ partition $S_{t-1}$, as every vector in $S_{t-1}$ agrees with exactly one value of $\Delta(Q_t,H)$. 
	
	We call the subset of $S_{t-1}$ corresponding to a response $r$ a \textit{solution subset}, formally defined as 
	\begin{equation*}
		B_t(r) = \left\{h\in S_{t-1}\mid \Delta(Q_t, h) = r\right\},
	\end{equation*}
	where $S_0$ is the set of all permutations of $[n]$.
	
	Following this notation, we see that $S_t = B_{t}(\Delta(Q_t,H))$.
	
	It will be useful to know the sizes of the sets $B(r,Q) := \{h | \Delta(h,Q) = r\}$. The number of elements in $B(r,Q)$ is equal to the number of ways to choose $r$ indices that are fixed points with respect to the query sequence $Q_1$ multiplied by the number of ways to permute the remaining $n-r$ colors without any fixed points. Hence we have
	\begin{equation*}
		|B(r,Q)| = \binom{n}{r}D(n-r).
	\end{equation*}
	So we may define $B(r) := B(r,Q)$ as the right-hand side is independent of $Q$. By definition, $B_t(r)$ for a fixed $Q_t$ is a subset of $B(r,Q_t)$, and so
	\begin{equation*}
		|B_t(r)|\le |B(r)| = \binom{n}{r}D(n-r).
	\end{equation*}
	
	\subsection{Proof of Theorem 1}
	Continuing, we make use of two technical lemmas. The first bounds the sums of sizes of the subsets defined above:
	\begin{lemma}  For any positive integer $n$, we have:
		\begin{equation*}
			\sum_{i=x}^n\binom{n}{i}D(n-i) \le \frac{n!}{x!}.
		\end{equation*}
	\end{lemma}
	\begin{proof}
		We give a combinatorial proof. The left-hand side denotes the number of permutations of an $n$-element vector which have at least $x$ fixed points. The right-hand side denotes the number of ways to choose $x$ fixed points and simply permute the rest of the vector. This counts all vectors with at least $x$ fixed points at least once (and over-counts by some margin) so the inequality holds.
	\end{proof}
	This will allow us to prove a bound on the worst-case size of $|S_t|$. Note that for a fixed strategy and hidden vector $H$, every query $Q_t$, response $r_t$, and remaining solution subset $S_t$ are completely determined.
	\begin{lemma} For any fixed deterministic guessing strategy and $C_n < n$, there is at least one choice of hidden vector such that 
		\begin{equation*}
			\frac{|S_t|}{n!}\ge \frac{C_{n}! - (H_{C_{n}+t} - H_{C_{n}})}{(C_n+t)!},
		\end{equation*}
		for all $0\le t\le n-C_{n}$, where $H_n=\sum_{i = 1}^n\frac{1}{i}$ is the $n\nth$ harmonic number.
	\end{lemma}
	We prove Lemma 2 at the end of the section.\\
	\begin{proof}[Proof of Theorem 1]
		Apply Lemma 2 for $t = n-C_n$. Then we have
		\begin{equation*}
			|S_{n-C_{n}}|  \ge n!\left(\frac{C_{n}! - (H_n - H_{C_{n}})}{n!}\right)
			= C_{n}! - (H_n-H_{C_{n}}).
		\end{equation*}
		For any $C_{n}$ such that $C_{n}! - (H_n-H_{C_{n}}) > 1$, the above bound gives $|S_{n-C_n}| > 1$. This would mean that after $n-C_n$ guesses of any strategy, the remaining solution set is not necessarily reduced to a single element after $n-C_n$ guesses, allowing us to state $f(n) > n-C_n$.
		
		Noting that $H_n$ is asymptotic to $\log n$ and both grow to infinity, as long as $\log n = o(C_n!)$ we will eventually have that $C_n!-H_n > 1$. Since we have, for example, that $\log x = o((\log \log x)!)$, we have that with $C_n = \lceil \log \log n \rceil$ the above inequality will eventually be satisfied.
		
		In conclusion, when $n$ is sufficiently large, the minimum number of remaining
		possible solutions after $n - \lceil \log\log n \rceil$ guesses is at least
		\begin{equation*}
			S_{n - \lceil \log\log n \rceil} 
			\ge (\log\log n)! - (H_n - H_{\log\log n})
			> 1.
		\end{equation*}
		Thus there is no strategy that can identify any hidden sequence in fewer than $n-\log\log n$ turns, which concludes the proof.
	\end{proof}

	\subsection{Proof of Lemma 2}
	\begin{proof}[Proof] 
		Recall that $S_t$ is the set of sequences that match the responses to the first $t$ questions of some fixed deterministic guessing strategy, given some hidden code $H$. Since all queries are possible when 0 questions have been asked, we have $|S_0| = n!$.
		
		From our definition of $B_t(r)$ above, the worst-case size of $S_t$ given $S_{t-1}$ is
		\begin{equation*}
			\max_{r\in\{0,1,\ldots,n\}}|B_{t}(r)|.
		\end{equation*}
		The minimum possible value of this maximum occurs when the subsets partition $S_{t-1}$ as evenly as possible. We know that $|B_t(r)| \leq |B(r)|$, which is easily seen to be decreasing in $r$. 
		In the optimal distribution of this type, some subsets of higher index will have size equal to their upper bound $B(r)$, while the rest will be partially filled to some fixed amount. So with $|S_{t-1}|$ solutions remaining, there is an optimal $x$ such that completely filling subsets $x$ through $n$ and splitting the remaining solutions among subsets 0 through $x-1$ will give us this best distribution, and therefore a lower bound on $|S_t|$ in the worst case.
		
		Hence, for some optimal value of $x$,
		\begin{equation}
			|S_t| \ge \frac 1x \left(|S_{t-1}|-\sum_{i = x}^{n}B(r)\right) = \frac 1x \left(|S_{t-1}|-\sum_{i = x}^{n}\binom{n}{i}D(n-i)\right).
		\end{equation}
		In fact, $x$ is `optimal' precisely in that it maximizes the right-hand side of this inequality, and so the inequality in fact holds for all $x$.
		
		We apply Lemma 1 to bound the rightmost term and get 
		\begin{equation}\label{recur}
			|S_t| \ge \frac{1}{x}\left(|S_{t-1}|-\frac{n!}{x!}\right).
		\end{equation}
		Given this recurrence, we proceed to prove Lemma 2 by induction.
		With $t=0$, we have $|S_0| = n!$. Then
		\begin{equation*}
			1 =\frac{|S_0|}{n!}\ge \frac{C_{n}! - (H_{C_{n}}-H_{C_{n}})}{C_{n}!} = 1,
		\end{equation*}
		and the inequality is satisfied. \\
		Now we move to the general case.
		Recalling that (\ref{recur}) holds for all $x$, we
		let $x=t+C_{n}$, giving
		\begin{equation*}
			\frac{|S_t|}{n!}
			\ge \frac{1}{C_{n}+t}\left(\frac{|S_{t-1}|}{n!}
			-\frac{1}{(C_{n}+t)!}\right).
		\end{equation*}
		Assuming the lemma inductively for $t-1$, we obtain
		\begin{align*}
			\frac{|S_t|}{n!}
			& \ge \frac{1}{C_{n}+t}\left(\frac{C_{n}!-(H_{C_{n}+t-1} - H_{C_{n}})}
			{(C_{n}+t-1)!} - \frac{1}{(C_{n}+t)!}\right)\\
			& \ge \left(\frac{C_{n}! - (H_{C_{n}+t-1} - H_{C_{n}}) - \frac{1}{C_{n}+t}}
			{(C_{n}+t)!}\right)\\
			& \ge \left(\frac{C_{n}! - (H_{C_{n}+t} - H_{C_{n}})}{(C_{n}+t)!}\right),
		\end{align*}
		which completes the induction.
	\end{proof}
	\textbf{Update:} El Ouali, Glazik, Sauerland, and Srivastav \cite{OS16} have announced an improvement of the lower bound in Theorem \ref{permutationGameTheorem} from $n-\log\log(n)$ to $n$, and an extension to $k > n$ with a lower bound of $k$.
	\section{Linear Algebra and the Minimax Algorithm}
	We now turn to a general upper bound on all Mastermind variants and note its application to Knuth's minimax algorithm in particular.

	We will represent an arbitrary query or hidden vector as a $(0,1)$-vector $A \in \R^{nk}$ in the following manner: 
	$$
	A_{in+j} = \begin{cases}
	1 & \text{this guess/solution assigns the $i^\text{th}$ spot the $j^\text{th}$ color} \\
	0 & \text{otherwise},
	\end{cases}
	$$
	where $0 \leq i \leq n-1$ and $0 \leq j \leq k-1$. As such, each set of indices $A_{in},\dots,A_{in+k-1}$ will have exactly one 1, as the $i^\text{th}$ position is exactly one color.
	
	With this notation, the black-peg distance becomes the dot product of the guess and the hidden vector, as there will be a contribution to the dot product exactly when both vectors have a one in the same spot, i.e. there is the same color in the same spot of both vectors. 
	
	The goal of Mastermind is then to find the unique valid $(0,1)$-vector such that its dot product with the hidden vector is $n$. By linearity of the dot product, once the codebreaker has queried any set of queries, it is possible to deduce the black-peg response to any linear combination of those queries as the corresponding linear combination of their responses. It follows that a winning strategy is simply to query a basis for the span of the set of valid queries, which, as a subspace of $\R^{nk}$, is of size at most $nk$. Note that this strategy does not make use of adaptive feedback, of white-peg responses, or the condition on repeated colors, and hence applies to every variant of Mastermind studied in this paper.
	
	This logic also allows us to bound the minimax algorithm (Definition \ref{minimax def}). 
	\subsection{Proof of Theorem 2}
	The theorem can be easily reduced to the following lemma:
	\begin{lemma}
		At each turn, if the minimax algorithm has not determined the hidden vector, it guesses a vector that is linearly independent of its previous guesses.
	\end{lemma}
	
	\begin{proof}
		
		From the reasoning above, when the hidden vector is not known it must be linearly independent from the previous queries (and indeed, so must the entire remaining solution set). At this point, guessing a vector $q$ that is a linear combination of the previous guesses returns no new information, so the size of the remaining solution set remains unchanged and the guess $q$ receives a score of $|S_{t-1}|$. If we can find a new query $r$ that is guaranteed to eliminate at least one vector from the remaining solution set, the minimax algorithm will choose $r$ over $q$. Choosing any vector from the remaining solution set will do the trick; a black-peg response of $n$ will cut $S_t$ down from at least two vectors to the single vector $r$, whereas any other response will certainly eliminate $r$ from $S_t$. Thus, the minimax algorithm will never make a guess that is linearly dependent on its previous guesses.
	\end{proof}
	Theorem \ref{minimaxTheorem} follows as an immediate corollary from this lemma: Since these vectors are members of $\R^{nk}$, the minimax algorithm can make at most $nk$ linearly independent guesses. After that, the minimax cannot make a linearly independent guess and by the above lemma, the hidden vector must be uniquely determined.
	
	\section{Non-Adaptive Variants}
	This section follows closely the reasoning  in \cite{DS13} as they analyze non-adaptive games.  We perform an analysis of the black-peg, non-adaptive, no-repeats variant. We make use of \textit{(Shannon) Entropy}, defined for a random variable $X$ to be:
	\begin{equation*}
		H(X)
		:= \sum_{x\in \text{Domain}(X)}\mathbb{P}[X=x]\cdot (-\log_2(\mathbb{P}[X=x])).
	\end{equation*}
	Entropy is \textit{subadditive}, 
	that is, if $X_1, X_2, \ldots, X_n$ are random variables and $X = (X_1,\dots,X_n)$ is the random vector containing the $X_i$ as entries, then
	\begin{equation}\label{subadditivityDefinition}
		H(X)\le \sum_{i=1}^{n}H(X_i).
	\end{equation}
	
	\subsection{Proof of Theorem \ref{nonAdaptiveTheorem}}
	\begin{proof}
		Consider a set $\{Q_1, Q_2, \ldots, Q_s\}$ of $s$ query sequences such that any possible hidden sequence may be uniquely determined by the responses $\{\Delta(Q_t, H)\}$. That is, if two possible hidden vectors $H$ and $H'$ satisfy $\Delta(Q_i, H) = \Delta(Q_i, H')$ for all $i$, then $H = H'$.
		
		Let the hidden vector $Z$ be sampled uniformly at random from the set of $k!/(k-n)!$ possibilities. Then the responses $Y_i = b(Z, q_i)$ are now random variables, and the vector $Y = (Y_1, Y_2, \ldots, Y_s)$ is also a random variable. By our assumptions, $Y$ always uniquely determines, and is uniquely determined by, $Z$, and so $H(Z) = H(Y)$. Since $Z$ is a random variable with $k!/(k-n)!$ outcomes of equal probability, we compute
		\begin{equation}\label{equalEntropy}
			H(Y) = H(Z) = \log_2\left(\frac{k!}{(k-n)!}\right).
		\end{equation}
		Combining this with (\ref{subadditivityDefinition}) yields:
		\begin{equation}\label{entropySum}
			\log_2\left(\frac{k!}{(k-n)!}\right)\le\sum_{i=1}^s H(Y_i).
		\end{equation}
		Now we bound $H(Y_i)$ by some absolute constant. By definition, $Y_i$ is precisely the black-peg distance between $Q_i$ and $Z$. Then
		\begin{equation}\label{entropyDef}
			H(Y_i)=-\sum_{x=0}^n \mathbb{P}[Y_i=x]\cdot\log_2(\mathbb{P}[Y_{i}=x]).
		\end{equation}
		The quantity $\mathbb{P}[Y_i=x]$ is probability that $Z$ is a solution vector with $x$ fixed points with respect to the query $q_i$. Using the terminology from section \ref{Buckets Definition}, this is the probability that $Z$ is in solution subset $B(x)$.
		We use a simple bound on $|B(x)|$: this counts the number of permutations with exactly $x$ fixed points, which is fewer than the number of ways to choose $x$ fixed points and permute the other colors arbitrarily. Thus:
		\begin{align*}
			\mathbb{P}[Y_i = x] = \frac{|B(x)|}{\left(\frac{k!}{(k-n)!}\right)}
			\leq \frac{\binom{n}{x}\frac{(k-x)!}{(k-n)!}}{\left(\frac{k!}{(k-n)!}\right)}
			= \frac{1}{x!}\cdot\frac{n(n-1)\cdots(n-x+1)}{k(k-1)\cdots(k-x+1)}
			\leq \frac{1}{x!}.
		\end{align*}
		We can substitute this upper bound into  (\ref{entropyDef}) for all $x<1/e$, because $-\alpha\log_2\alpha$ is an increasing function on that domain. For the first three values of $x$, we instead use the trivial upper bound $-\alpha\log_2\alpha \le  1/(e\log 2)$. Then
		\begin{align*}
			H(Y_i) &\leq \frac{3}{e\log 2}+\sum_{x=3}^n-\frac{1}{x!}\cdot\log_2\left(\frac{1}{x!}\right)
			\le \frac{3}{e\log 2}+\sum_{x=3}^\infty \frac{\log_2(x!)}{x!}
			<3.
		\end{align*}
		Combining this with (\ref{entropySum}) gives $H(Y) \leq 3s$, and since $H(Z) = H(Y)$ we have $H(Z) \leq 3s$. Substituting in (\ref{equalEntropy}) as a lower bound for $H(Z)$ and solving for $s$ gives
		\begin{equation*}
			s \geq \frac{1}{3}\log_2\left(\frac{k!}{(k-n)!}\right).
		\end{equation*}
		
		We can show that the right-hand side is $\Omega(n\log k)$. This is immediate for large $k$ relative to $n$ by bounding the ratio by $(k-n)^n$; for small $k$ (e.g. $k \leq 2n$) we bound the ratio by $n!$ and the claim follows from Stirling's approximation. So this gives us a lower bound of $\Omega(n \log k)$ turns for any non-adaptive strategy for Mastermind with no repeats and black-peg responses.
	
	\end{proof}
	\subsection{A Note on Non-Adaptive Variants with Repeats}
	Consider non-adaptive variants of Mastermind in which repetitions are allowed. When $k=n$, Theorem 13 of \cite{DS13} guarantees the existence of a set of $O(n \log n)$ queries which uniquely identify any hidden sequence $H$. When $k > n$, one can simply extend $H$ and all queries $Q_t$ by $k - n$ ``auxiliary'' positions. We fill these auxiliary positions with arbitrary colors, and adjust the codemaker's responses accordingly. Applying Theorem 13 of \cite{DS13} now guarantees the existence of a set of $O(k \log k)$ queries which will uniquely identify any hidden sequence.
	
\section{Further Work}
\begin{problem}
	In all adaptive variants of Mastermind, the best lower bounds are asymptotic to $n$. Can these lower bounds be improved?
\end{problem}
\begin{problem}
	In the variant of Mastermind with repeats, the best known strategy is $O(n\log\log k)$ \textnormal{\cite{DS13}}, whereas in the variant without repeats, it is $O(n \log k)$ \textnormal{\cite{OS13}}. Is it possible to extend or modify the first strategy to apply to the no-repeats game?
\end{problem}
\begin{problem}
	Theorem \ref{DSTW nonadapt} can be extended to white-peg responses as well by noting that the optimal white-peg strategy may be converted into a black-peg strategy losing a constant multiplicative factor by adding $k$ guesses, each composed entirely of pegs of a single color. The same construction does not apply to extend Theorem \ref{nonAdaptiveTheorem}. Can a lower bound be established in this case?
\end{problem}
\begin{problem}
	In non-adaptive variants, the best lower bounds are $O(n \log k)$, whereas upper bounds are either $O(k \log k)$ or $nk$. Can these bounds be brought closer together?
\end{problem}
	
	\section{Acknowledgments}
	We would like to thank Daniel Montealegre for supervising and assisting our work throughout the summer, and Nathan Kaplan for both creating and guiding our project and for his invaluable assistance in editing this paper. 
	We would also like to thank Sam Payne and the Summer Undergraduate Math Research at Yale program for organizing, funding, and supporting this project. SUMRY is supported in part by NSF grant CAREER DMS-1149054.

	\clearpage
	
	\bibliography{mybibfile}
	
\end{document}